\documentclass[12pt]{amsart}
\usepackage{amssymb,amsmath}
\usepackage{enumerate}
\usepackage{mathrsfs}
\usepackage{xcolor}

\newtheorem{theorem}{Theorem}[section]
\newtheorem{proposition}[theorem]{Proposition}
\newtheorem{lemma}[theorem]{Lemma}
\newtheorem{corollary}[theorem]{Corollary}
\theoremstyle{definition}
\newtheorem{definition}[theorem]{Definition}

\newcommand{\clb}{\mathscr{B}}

\newcommand{\clm}{\mathcal{M}}

\newcommand{\raro}{\rightarrow}

\newcommand{\vp}{\varphi}

\newcommand{\bd}{\mathbb{D}}
\newcommand{\bt}{\mathbb{T}}
\newcommand{\bc}{\mathbb{C}}

\title[Automorphisms and Generalized projections]{Automorphisms and Generalized projections on spaces of analytic functions}

\author[Maurya]{Rahul Maurya}
\address{Department of Mathematics, Indian Institute of Science, Bangalore, 560012, India}
\email{rahulmaurya7892@gmail.com}

\author[Sarkar]{Jaydeb Sarkar}
\address{Indian Statistical Institute, Statistics and Mathematics Unit, 8th Mile, Mysore Road, Bangalore, 560 059, India}
\email{jaydeb@gmail.com, jay@isibang.ac.in}

\author[Sensarma]{Aryaman Sensarma}
\address{Indian Statistical Institute, Statistics and Mathematics Unit, 8th Mile, Mysore Road, Bangalore, 560 059, India}
\email{aryamansensarma@gmail.com}
\subjclass{46B20, 30H05, 46J15, 47L20, 46J10,46E15, 47B38}
\keywords{Bounded analytic functions, $H^p$-spaces, automorphisms, Neil algebra, generalized tri-circular projections}

\numberwithin{equation}{section}

\begin{document}

%\today

\begin{abstract}
We present complete classifications of automorphisms of two closed subalgebras of the bounded analytic functions on the open unit disc $\mathbb{D}$, namely, the subalgebra of functions vanishing at the origin, and the subalgebra of functions whose first derivative vanishes at the origin. The later subalgebra is known as the Neil algebra. We also characterize generalized tri-circular projections on $H^{p}(\mathbb{D})$ and $H^{p}(\mathbb{D}^2)$, $1\leq p \leq \infty$, $p\neq 2$.
\end{abstract}
\maketitle

%\tableofcontents

\section{Introduction}

Let $\bd$ denote the open unit disc in the complex plane. Denote by $H^\infty(\bd)$ the commutative Banach algebra of all bounded analytic functions on $\bd$ with
\[
\|\vp\|_\infty = \sup \{|\vp(z)|: z \in \bd\} \qquad (\vp \in H^\infty(\bd)).
\]
This is one of the most important non-separable Banach algebras, with a variety of applications in function theory, operator theory, and operator algebras. The disc algebra $A(\bd)$ is another important space among the classical separable Banach algebras. Recall that
\[
A(\bd) = H^\infty(\bd) \cap C(\mathbb{T}).
\]
In general, given a domain $\Omega$ in $\mathbb{C}$, we denote by $A(\Omega)$ the space of analytic functions that extends continuously to the boundary of $\Omega$. A classic result of L. Bers \cite{B} serves as the starting point for our discussion: Two domains $\Omega_1$ and $\Omega_2$ are conformally equivalent if and only if there is an automorphism between $A(\Omega_1)$ and $A(\Omega_2)$. Recall that two domains are conformally equivalent if there exists an angle preserving bijective analytic map between them. This is the same as asserting that the domains being biholomorphically equivalent. In what follows,  we will refer to an \textit{automorphism} as a linear and algebra isomorphism.

Returning to the special case when $\Omega_1 = \Omega_2 = \bd$, an isomorphism $T: A(\bd) \raro A(\bd)$ is precisely given by (again, see Bers \cite{B})
\[
Tf = f \circ \tau \qquad (f \in A(\bd)),
\]
for some conformal map $\tau$ on $\bd$ (that is, $\tau \in Aut(\bd)$). Moreover, for $H^\infty(\bd)$, we have the same conclusion (a particular case of Rudin \cite{R}): Isomorphisms between $H^\infty(\bd)$ are induced by conformal mappings as described above. As Bers and Rudin noted, Chevalley and Kakutani's earlier work also inspired this development. See also \cite{Ro,Su} for more classical advancement.

In this paper, we examine automorphisms of two important subalgebras of $H^\infty(\bd)$ namely $H_0^{\infty}(\bd)$ and the Neil algebra. First, we recall that
\[
H_0^{\infty}(\mathbb{D})=\{f \in H^{\infty}(\mathbb{D}): f(0)=0\}.
\]
Like $H^{\infty}(\mathbb{D})$, the above space is crucial to functional analysis. For instance, see Lomonosov \cite{Lomonosov} for counterexamples in the context of the Bishop-Phelps-Bollob\'{a}s theorem. Next, we recall the Neil Algebra ${H}_{1}^{\infty}(\mathbb{D})$ (see \cite{DPRS}):
\[
{H}_{1}^{\infty}(\mathbb{D}) = \{f\in {H}^{\infty}(\mathbb{D}): f'(0) = 0\}.
\]
This space is commonly used to test classical theories such as the interpolation problem, corona theorem, commutant lifting theorem, and invariant subspaces, to name a few. The following is a summary of the main results concerning the automorphisms of $H_0^{\infty}(\mathbb{D})$ and $H_1^{\infty}(\mathbb{D})$ (see Theorems \ref{thm: auto of H0} and \ref{T2}). For any Banach space $X$, the Banach algebra of bounded linear operators on $X$ is denoted by $\clb(X)$.

\begin{theorem}
Let $X = H^\infty_0(\bd)$ or $H^\infty_1(\bd)$, and let $T : X \raro X$ be a map. Then $T$ is an automorphism if and only if there exists $\theta \in \mathbb{R}$ such that
\[
(Tf)(z) = f(e^{i\theta} z),
\]
for all $f \in X$ and $z \in \bd$.
\end{theorem}

Basically, this shows that $H^\infty_0(\bd)$ and $H^\infty_1(\bd)$'s automorphisms are simple (or trivial). Along the way, we prove that (see Theorem \ref{T1}) automorphisms of $H^{\infty}(\bd)$ preserve inner functions. We also prove that (see Corollary \ref{C3}): If $T$ is a surjective linear isometries of $H_1^{\infty}(\mathbb{D})$, then there exist $\alpha \in \bt$ and $\theta \in \mathbb{R}$ such that
\[
Tf(z) = \alpha f(e^{i\theta} z),
\]
for all $f \in H_1^{\infty}(\mathbb{D})$ and $z \in \bd$. The main idea of the proof of the preceding fact is straightforward. It all comes down to showing that $H_1^{\infty}(\mathbb{D})$ can be represented as a uniform algebra.

The second goal of this paper is to describe the structure of generalized tri-circular projections on $H^p(\bd)$ and $H^p(\bd^2)$, $1 \leq p \leq \infty$. The notion of generalized tri-circular projection comes from the ideas of bicircular projections, which also connect the structure of surjective isometries on Banach spaces. We begin by recalling the definition of projections on Banach spaces in order to be more precise. Given a Banach space $X$ (here all Banach spaces are over $\mathbb{C}$), a bounded linear operator $P$ on $X$ is called a \textit{projection} if
\[
P^2 = P.
\]
A projection $P \in \clb(X)$ is called \textit{bicircular projection} if $P+\lambda (I-P)$ is a surjective linear isometry for some $\lambda \in \mathbb{T}$, and it is called \textit{generalized bicircular projection} if there exists $\lambda \in \mathbb{T} \setminus \{1\}$ such that $P+\lambda (I-P)$ is a surjective linear isometry \cite{GBP}.

Projections are fundamental building blocks for more complex operators, but depending on the Banach spaces, they can be adequately complex. On the other hand, surjective linear isometries often help explain operators on Banach and Hilbert spaces \cite{FJ}. Given the preceding and subsequent definitions, we note that, despite the benefits of using surjective linear isometries to construct specific projections, which we will also discuss in this paper, the mechanism itself can be computationally intensive.

Generalized bicircular projections are fully described in some classical Banach spaces. For instance, finite dimensional Banach spaces with respect to various $G$-invariant norms \cite{GBP}, minimal ideals of operators \cite{GBP1}, $JB^*$-triples \cite{GBP3}, certain Hardy spaces \cite{GBP4}, $L^p$-spaces, $1 \leq p< \infty$, $p \neq 2$ \cite{Lin}, etc.

Generalized tri-circular projections, a finer concept of projections, were introduced earlier in \cite{AD}. The present definition, however, is derived from \cite[Definition 1.4]{CJD}.

\begin{definition}
A nonzero projection $P$ on a Banach space $X$ is said to be a generalized tri-circular projection if there exist distinct scalars $\lambda,\mu \in \mathbb{T} \setminus \{1\}$ and nonzero projections $Q, R \in \clb(X)$ such that $P \oplus Q \oplus R = I$ and $P + \lambda Q + \mu R$ is a surjective linear isometry.
\end{definition}

Generalized tri-circular projections and related topics have recently received increased attention. This also has to do with questions about projections and isometries on Banach spaces. Generalized tri-circular projections are completely characterized for some known spaces: $C(X)$ \cite{AD}, $\mathbb{C}^n$ and $M_n(\mathbb{C})$ \cite{AD1}, minimal norm ideals on operator algebra $\mathcal{B}(H)$ \cite{CJD}, $JB^*$-triple \cite{G3P1}, Hilbert $C_0(X)$-modules \cite{ILW}, and the Banach spaces of functions of bounded variation and of absolutely continuous functions \cite{MH1}.

Generally, the solution to a set of projection-related equations is used to classify generalized tri-circular projections. Thanks to \v{C}uka and Ili\v{s}evi\'{c}'s work \cite{CJD} (also see Lemma \ref{lemma: GTP class}), we employ identical methodologies. The classification of generalized tri-circular projections on $H^{p}(\mathbb{D})$ (see Theorem \ref{Th1}) appears to be similar to that of other classes of classical Banach spaces obtained earlier. However, the answer for $H^{p}(\mathbb{D}^2)$ (see Theorem \ref{Th2}) differs significantly, which perhaps highlights the complexity of several complex variables. We also remark that the later result is the first instance of generalized tri-circular projections in several variables.

The remaining part of the paper is organized as follows. Section \ref{sec: aut of H0} deals with the classification of automorphisms of $H^\infty_0(\bd)$. We also prove that automorphisms of $H^{\infty}(\bd)$ preserve inner functions. Section \ref{sec: aut H 1} studies the Neil algebra ${H}_{1}^{\infty}(\mathbb{D})$ and classifies automorphisms of $H^\infty_1(\bd)$. We also classify surjective linear isometries of $H_1^{\infty}(\mathbb{D})$. Section \ref{sec: GP on D} serves as the starting point for the second half of this paper. Here we characterize generalized tri-circular projections on $H^{p}(\mathbb{D})$, $1\leq p \leq \infty$, $p\neq 2$. The final section, Section \ref{sec: gen tri 2-var}, classifies the generalized tri-circular projections on $H^{p}(\mathbb{D}^2)$, $1\leq p \leq \infty$, $p\neq 2$.

\section{Automorphisms of $H_0^{\infty}(\bd)$}\label{sec: aut of H0}

The purpose of this section is to classify the automorphisms of $H^\infty_0(\bd)$. We begin by addressing a natural question that may be familiar to experts. However, we were unable to locate it in the literature. The proof uses a two-variables function theoretic result.

\begin{theorem}\label{T1}
Automorphisms of $H^{\infty}(\bd)$ preserve inner functions.
\end{theorem}
\begin{proof}
Let $T$ be an automorphism of $H^{\infty}(\bd)$, and let $\vp \in H^\infty(\bd)$ be an inner function. Assume, contrary to the desired conclusion, that $T\vp$ is not inner. There exists $\epsilon > 0$ such that $m(K)>0$, where
\[
K = \{z \in \mathbb{T} : |T\vp(z)|< 1 - \epsilon\}.
\]
Define the function $T\vp \otimes 1 \in H^{\infty}(\mathbb{T}^2)$ by
\[
(T\vp \otimes 1)(z_1, z_2):= (T\vp) (z_1),
\]
for all $(z_1, z_2) \in \mathbb{T}^2$. Set $K_1 = K \times \mathbb{T}$. Since
\[
|(T\vp \otimes 1)(z_1, z_2)| < 1 - \epsilon,
\]
for all $(z_1, z_2) \in K_1$, it follows that $m(K_1)>0$, and hence $T\vp \otimes 1$ is not inner. Now, in view of $K_1$, there exists $\psi \in H^{\infty}(\mathbb{T}^2)$ such that (cf. \cite[Theorem 5.9, p. 297]{KH})
\[
|\psi (z_1, z_2)| =
\begin{cases}
1 & \mbox{if } (z_1, z_2) \in K_1
\\
1 - \epsilon & \mbox{if } (z_1, z_2) \in \mathbb{T}^2 \setminus K_1.
\end{cases}
\]
Set
\[
K_2 = \{z_1 \in \mathbb{T} : (z_1, z_2)\in K_1 \text{ for some } z_2 \in \mathbb{T}\},
\]
and define a function $\psi_1\in L^{\infty}(\mathbb{T})$ by
\[
\psi_1(z_1) = \psi(z_1, z_2) \qquad (z_1 \in \mathbb{T}).
\]
Clearly, $|\psi_1(z_1)| = |\psi(z_1, z_2)| = 1$ on $K_2$, and $|\psi_1(z_1)| = |\psi(z_1, z_2)| = 1 - \epsilon $ on $\mathbb{T}\setminus K_2$. Since $\psi\in {H}^{\infty}(\mathbb{D}^2)$, we have the power series expansion
\[
\psi_1(z_1) = \psi(z_1, z_2) = \sum_{n, m = 0}^{\infty}a_{nm} z_1^nz_2^m.
\]
Since
\[
\sum_{n, m = 0}^{\infty}|a_{nm}z_2^m|^2 \leq \sum_{n, m = 0}^{\infty}|a_{nm}|^{2}|z_2^m|^{2} = \sum_{n, m = 0}^{\infty}|a_{nm}|^{2}< \infty,
\]
we conclude that $\psi_1\in H^{2}(\mathbb{T})$, and hence $\psi_1\in {H}^{\infty}(\mathbb{T})$. Now there exists $g\in {H}^{\infty}(\mathbb{T})$ such that $Tg = \psi_1$. Then $\|M_{\varphi}g\| = \|\varphi g\|_{\infty} = 1$, but
\[
\|T(\varphi g)\|_{\infty} = \|T(\varphi) T(g)\|_{\infty} = \|(T\varphi)\psi_1\|_{\infty} \leq 1 - \epsilon,
\]
which is a contradiction. Thus $T\varphi$ is an inner function.
\end{proof}

Next, our goal is to prove that automorphisms of $H^{\infty}_0(\bd)$ preserve inner functions. The following simple and general observation is crucial. In view of $f = f(0)+ (f - f(0))$ for all $f \in H^\infty(\bd)$, we write the Banach space direct sum as
\begin{equation}\label{eqn: decomp of H infty}
H^{\infty}(\mathbb{D}) = \mathbb{C} \dotplus H^{\infty}_{0}(\mathbb{D}).
\end{equation}
Fix an automorphism $T$ of $H^\infty_0(\bd)$, and define $X: H^{\infty}(\bd) \raro H^{\infty}(\bd)$ by
\begin{equation}\label{eqn: decomp of X}
X(\alpha + \beta f) = \alpha + \beta T^{-1}f,
\end{equation}
for all $\alpha, \beta \in \mathbb{C}$ and $f\in H^{\infty}_{0}(\mathbb{D})$, and claim that $X$ is an automorphism. Clearly, $X\vert_{H^{\infty}_{0}(\mathbb{D})} = T^{-1}$ and $X(1) = 1$. A routine computation shows that $X$ is linear and multiplicative. We check, for instance, the linearity of $X$: If $\alpha_i, \beta_i, \gamma \in \bc$ and $f_i \in H_0^\infty(\bd)$, $i=1,2$, then
\[
\begin{split}
X((\alpha_1 + \beta_1f_1) + \gamma(\alpha_2 + \beta_2f_2)) & = X(\alpha_1 + \gamma \alpha_2 + \beta_1f_1 + \gamma \beta_2f_2)
\\
& = \alpha_1 + \gamma \alpha_2 + T^{-1}(\beta_1f_1 + \gamma \beta_2f_2)
\\
& = \alpha_1 + T^{-1}(\beta_1f_1)+ \gamma (\alpha_2 +  T^{-1}(\beta_2f_2))\\
&=X(\alpha_1 + \beta_1f_1) + \gamma X(\alpha_2 + \beta_2f_2).
\end{split}
\]
Now we show that $X$ is injective: let $\alpha + \beta T^{-1}f = 0$. If $\beta = 0$, then $\alpha = 0$. Therefore, assume that $\beta \neq 0$. This implies that $\alpha + \beta (T^{-1}f)(z) = 0$ for every $z\in \mathbb{D}$. Since $T^{-1}f(0)=0$, if follows that $\alpha = 0$. Then $T^{-1}f = 0$, and hence $f = 0$. To prove that $X$ is onto, assume that $g\in H^{\infty}(\mathbb{D})$. Since $T^{-1}$ is onto and $g - g(0) \in H^\infty_0(\bd)$, there exists $\tilde{g} \in  H^{\infty}_{0}(\mathbb{D})$ such that
\[
T^{-1}(\tilde{g}) = g - g(0).
\]\
Then $g = g(0) + T^{-1}(\tilde{g}) = X(g(0) + \tilde{g})$, which ends the proof of the claim

\begin{lemma}\label{lemma: 1-var inner pres}
Automorphisms of $H^{\infty}_0(\bd)$ preserve inner functions.
\end{lemma}
\begin{proof}
Fix an automorphism $T$ of $H^\infty_0(\bd)$, and consider $X$ as defined in \eqref{eqn: decomp of X}. Pick an inner function $\vp \in H^{\infty}_0(\bd)$, and assume on contrary that $T \vp \in H^{\infty}_{0}(\bd)$ is not inner. We know that $X\vert_{H^{\infty}_{0}(\mathbb{D})} = T^{-1}$, and hence by Theorem \ref{T1} we conclude that $X(T\vp) = T^{-1}(T\vp) = \vp$ is inner, which is a contradiction.
\end{proof}

We are now ready for characterizations of the automorphisms of $H^{\infty}_{0}(\mathbb{D})$. Our proof is in the lines of deLeeuw, Rudin, and Wermer \cite{DRW}.

\begin{theorem}\label{thm: auto of H0}
Let $T: H^\infty_0(\bd) \raro H^\infty_0(\bd)$ be a map. Then $T$ is an automorphism if and only if there exists $\theta \in \mathbb{R}$ such that
\[
(Tf)(z) = f(e^{i\theta} z) \qquad (f \in H^\infty_0(\bd), z \in \bd).
\]
\end{theorem}
\begin{proof}
The sufficient part is trivial. For the necessary direction, consider the inner function $\vp \in H^{\infty}_{0}(\mathbb{D})$ defined by $\vp(z) = z$, $z \in \bd$. Then $\tau:=T\vp$ is an inner function (see Lemma \ref{lemma: 1-var inner pres}). Note that $\tau(0) = 0$ forces that $\tau$ is non-constant, which, in turn, yields $|\tau|<1$ on $\mathbb{D}$ (cf. \cite[Theorem 2.2.10]{RAM}). Next, we fix $f\in H^{\infty}_{0}(\mathbb{D})$ and $z_0\in \mathbb{D}$. Since $(f^2 - f(\tau(z_0))f)(\tau (z_0)) = 0$ and $f(0) = 0$, there exist $g, \tilde{g} \in H^\infty(\bd)$ such that
\[
f^2 - f(\tau(z_0)) f = (\vp - \tau(z_0))g,
\]
and
\[
f = \vp \tilde{g}.
\]
Together, the equalities mentioned above imply that
\[
f^3 - f(\tau(z_0)) f^2 = (\vp - \tau(z_0))\vp g \tilde{g}.
\]
Again, in view of $f = \vp \tilde{g}$, this further leads to
\[
f^4 - f(\tau(z_0)) f^3 = (\vp^2 - \tau(z_0)\vp)\vp g \tilde{g}^2.
\]
Applying the operator $T$ to both sides of the above, we find
\[
Tf^4 - f(\tau(z_0)) Tf^3 = T(\vp^2 - \tau(z_0)\vp) T(\vp g \tilde{g}^2).
\]
Since $T(\vp) = \tau$, $T(\vp^2) = \tau^2$ (recall that $T$ is multiplicative) and $(\tau^2 - \tau(z_0) \tau)(z_0) = 0$, by evaluating at $z=z_0$, the above identity yields $(Tf^4)(z_0) = f(\tau(z_0))(Tf^3)(z_0)$. Since $z_0 \in \bd$ was arbitrary, we conclude that
\[
(Tf^4)(z) = f(\tau(z))(Tf^3)(z) \qquad (z \in \bd).
\]
By setting $T^{-1} \vp = \psi$, a similar computation yields the following identity:
\[
(T^{-1}f^4)(z) = f(\psi(z))(T^{-1}f^3)(z) \qquad (z \in \bd).
\]
Writing $\vp^4 = T(T^{-1}(\vp^4))$ and then applying the above identity to $f = T^{-1}\vp$, we find
\[
\begin{split}
\vp^4 & = T((T^{-1}\vp)^{4})
\\
& = (T^{-1}\vp \circ \tau) T(T^{-1} \vp)^3
\\
& = (\psi \circ \tau)T(T^{-1}\vp^3)
\\
& = (\psi \circ \tau) \vp^3.
\end{split}
\]
By the definition of $\vp$, it follows that $\psi \circ \tau(z) = z$, $z\in \mathbb{D}$. Similarly, using $T^{-1}(T\vp)^{4} = (T \vp \circ \psi) T^{-1}(T \vp)^3$, we find $\tau \circ \psi(z) = z$, $z\in \mathbb{D}$, which proves that $\tau$ is a conformal map of $\mathbb{D}$, hence there is a real number $\theta$ such that
\[
\tau(z) = e^{i\theta} z \qquad (z \in \bd).
\]
For a fixed $z_0 \in \bd$, we again observe that
\[
f \vp - f(\tau(z_0))\vp = (\vp - \tau(z_0))h,
\]
for some $h\in H^{\infty}(\mathbb{D})$. By multiplying both sides by $\vp^2$, if follows that
\[
f \vp^3  - f(\tau(z_0))\vp^3 = (\vp^2 - \tau(z_0)\vp) h\vp.
\]
Then $T((\vp^2 - \tau(z_0)\vp))(z_0) = 0$ implies that
\[
(Tf)(z) \tau^3(z) = f(\tau(z))\tau^3(z),
\]
for all $f\in H^{\infty}_{0}(\mathbb{D})$ and $z\in \mathbb{D}$. Using the equality $\tau(z) = e^{i\theta} z$, we conclude that
\[
Tf(z) e^{3i\theta} z = f(e^{i\theta} z)e^{3i\theta}z,
\]
for all $z\in \mathbb{D}$, and hence $Tf = f \circ \tau$. This completes the proof of the theorem.
\end{proof}

\section{Automorphisms of Neil Algebra}\label{sec: aut H 1}

In this section we characterise the automorphism of the Neil Algebra ${H}_{1}^{\infty}(\mathbb{D})$. We also classify surjective linear isometries of ${H}_{1}^{\infty}(\mathbb{D})$. Recall that
\[
H_{1}^{\infty}(\mathbb{D}) = \{f\in {H}^{\infty}(\mathbb{D}): f'(0) = 0\},
\]
is a closed Banach subalgebra of ${H}^{\infty}(\mathbb{D})$.

\begin{theorem}\label{T2}
Let $T : H^\infty_1(\bd) \raro H^\infty_1(\bd)$ be a map. Then $T$ is an automorphism if and only if there exists $\theta \in \mathbb{R}$ such that
\[
(Tf)(z) = f(e^{i\theta} z) \qquad (f \in H^\infty_1(\bd), z \in \bd).
\]
\end{theorem}
\begin{proof}
The sufficient condition holds trivially so we only need to prove the necessary condition. Let $f\in {H}^{\infty}_{1}(\mathbb{D})$. Clearly, $\lambda\in \overline{Ran(f)}$ if and only if $f - \lambda$ is not invertible in ${H}^{\infty}_{1}(\mathbb{D})$. Since $T$ is an automorphism, we get $f - \lambda$ is invertible if and only if $Tf - \lambda$ is invertible. Therefore
\[
\overline{Ran(f)} = \overline{Ran(Tf)}.
\]
Consider the identity function $id(z) = z$ for all $z \in \bd$. Then $id^2(z) = z^2$ and $id^3(z) = z^3$. Let $f_2:=T(id^2)$, and $f_3:=T(id^3)$. Since $T$ is an automorphism, $f_2, f_3 \neq 0$. Observe that
\[
T(id^6)=T(id^2)^3=f_2^3=T(id^3)^2=f_3^2.
\]
Let $z_0$ be a zero of $f_2$ of multiplicity $n$. Since $f_2^3 = f_3^2$, it follows that $z_0$ is a zero of $f_3^2$ of multiplicity $3n$ for some $n \geq 1$. In other words, $z_0$ is a zero of $f_3$ of multiplicity $3n/2 \in \mathbb{N}$, and hence $n$ is even. Therefore
\[
\tau := \dfrac{f_3}{f_2} \in \text{Hol}(\bd).
\]
Then, $f_3^2 = \tau^2 f_2^2 = f_2^3$, and hence $f_2=\tau^2$ outside the isolated zeros of $f_2$. By the identity theorem, $f_2=\tau^2$, and so $f_3=\tau^3$ on $\mathbb{D}$. Since $\tau^2=f_2$ is bounded, we conclude that $\tau \in H^\infty(\bd)$. Moreover, we have that $T(id^2) = \tau^2$ and $T(id^3) = \tau^3$. Now
\[
\overline{Ran(\tau^2)} = \overline{Ran(\tau^3)} = \overline{id^2} = \overline{id^3} = \overline{\mathbb{D}}.
\]
By the open mapping theorem, $\tau^2$ and $\tau^3$ are open maps. Thus $\tau^2(\mathbb{D})\subset \mathbb{D}$, and $\tau^3(\mathbb{D})\subset \mathbb{D}$, which implies that $\tau(\mathbb{D})\subset \mathbb{D}$. Next, we claim that $\tau$ is in $Aut(\bd)$. To this end, let $f\in {H}^{\infty}_{1}(\mathbb{D})$ and let $z_0\in\mathbb{D}$. Since $f - f(\tau(z_0))$ vanishes at $\tau(z_0)$, we have
\[
f - f(\tau(z_0)) = (id - \tau(z_0))g,
\]
for some $g \in H^\infty(\bd)$. Multiplying each side by $id^4$ and then applying $T$, we find
\[
T (f \, id^4) - f(\tau(z_0)) T (id^4) = (T (id^3) - T (\tau(z_0)id^2)) T(g \, id^2).
\]
In particular, if $z= z_0$ is arbitrary, then
\[
Tf(z_0)\tau^4(z_0) - f(\tau(z_0)) \tau^4(z_0) = (\tau^3(z_0) - \tau(z_0)\tau^2(z_0))T(g \, id^2)(z_0),
\]
which (after simplification) gives us $Tf(z_0)\tau^4(z_0) - f(\tau(z_0)) \tau^4(z_0) = 0$, and hence
\[
(Tf - f\circ \tau) \tau^4 = 0.
\]
Choose a neighbourhood $N(0,\delta)$ of $0 \in \mathbb{D}$ such that
\[
\overline{N(0,\delta)} \subset \mathbb{D}.
\]
Since $\tau^4 \not\equiv 0$ (note that $T$ is an algebra automorphism), $\tau^4$ has finitely many zeros in $N(0,\delta)$. Thus $Tf-f\circ \tau$ has infinitely many zeros in  $N(0,\delta)\subset \mathbb{D}$. By the identity theorem, we conclude
\[
Tf = f\circ \tau \qquad (f \in {H}^{\infty}_{1}(\mathbb{D})).
\]
Since $T^{-1}$ is also algebra automorphism, by the previous argument, there exists $\psi \in H^\infty(\bd)$ such that
\[
T^{-1}g = g \circ \psi   \qquad (g \in {H}^{\infty}_{1}(\mathbb{D})),
\]
and hence, for each $g \in {H}^{\infty}_{1}(\mathbb{D})$, we have
\[
g = T^{-1}(Tg) = T^{-1}(g \circ \tau) = g \circ (\tau \circ \psi) \qquad (g \in {H}^{\infty}_{1}(\mathbb{D})).
\]
In particular, if $g(z) = z^2$, then $z^2 = (\tau(\psi(z)))^2$, that is
\[
(z - \tau(\psi(z)))(z + \tau(\psi(z))) = 0.
\]
We claim that $\tau(\psi(z)) = z$ for every $z \neq 0$ in $\mathbb{D}$. To show this, first, we observe following the proof of the above equality that $g(z) = z^3$ implies that
\[
(z - \tau(\psi(z))(z^2 + (\tau(\psi(z))^2 + z \tau(\psi(z))) = 0.
\]
If $\tau(\psi(z_0)) = -z_0$ for some nonzero $z_0 \in \mathbb{D}$, then we get $2z_0^3 = 0$, which is a contradiction. This proves the claim that $\tau (\psi(z)) = z$ for every $z \neq 0$ in $\mathbb{D}$. Applying the identity theorem, we finally conclude that $\tau \circ \psi = id$. Using similar argument we get, $\psi \circ \tau = id$. Therefore, we conclude that $\tau$ is a conformal map. Since $f\circ \tau\in {H}^{\infty}_{1}(\mathbb{D})$, we have
\[
f'(\tau(0))\tau'(0) = 0.
\]
Since $\tau^{'}(0) \neq 0$, we get  $f'(\tau(0)) = 0$ for every $f\in{H}^{\infty}_{1}(\mathbb{D})$. Choose, for instance, $f(z) = \frac{z^2}{2}$, and conclude that $\tau(0) = 0$. This completes the proof.
\end{proof}

Now we turn to characterizations of surjective linear isometries of $H_1^{\infty}(\mathbb{D})$. The result is largely a consequence of the fact that $H_1^{\infty}(\mathbb{D})$ is uniform algebra. Denote by $\clm(H^\infty_1(\bd))$ the maximal ideal space of $H^\infty_1(\bd)$. It is clear that $\clm(H^\infty_1(\bd))$ is a complex object, and that the structure of $\clm(H^\infty_1(\bd))$ will be a key factor in many questions regarding the Banach algebra $H^\infty_1(\bd)$. We apply the basic structure of $\clm(H^\infty_1(\bd))$ to prove that $H_1^{\infty}(\mathbb{D})$ is a uniform algebra (just as in the case of $H^{\infty}(\mathbb{D})$).

\begin{corollary}\label{C3}
Let $T$ be a surjective linear isometries of $H_1^{\infty}(\mathbb{D})$. Then there exist $\alpha \in \bt$ and $\theta \in \mathbb{R}$ such that
\[
Tf(z) = \alpha f(e^{i\theta} \tau(z)) \qquad (f \in H_1^{\infty}(\mathbb{D}), \, z \in \bd).
\]
\end{corollary}
\begin{proof}
Consider the Gelfand map
\[
\Gamma: H_1^{\infty}(\mathbb{D}) \longrightarrow C(\clm(H^\infty_1(\bd))),
\]
defined by
\[
\Gamma f = \hat f,
\]
where $\hat{f}(\vp)=\vp(f)$ for all $\vp \in \clm(H^\infty_1(\bd))$ and $f \in H_1^{\infty}(\mathbb{D})$. For each $f \in H_1^{\infty}(\mathbb{D})$, we compute (just as in the case of $H^{\infty}(\mathbb{D})$)
\[
\displaystyle\|\hat{f}\|=\sup_{\vp \in \mathcal{M}(H_1^{\infty}(\mathbb{D}))}|\hat{f}(\vp)|=\sup_{\vp \in \mathcal{M}(H_1^{\infty}(\mathbb{D}))}|\vp(f)|\leq \|f\|.
\]
On the other hand, for each $\vp \in \clm(H^\infty_1(\bd))$, we have
\[
\displaystyle\|\hat{f}\|\geq \sup_{\vp_{\lambda} \in \mathcal{M}(H_1^{\infty}(\mathbb{D}))}|\hat{f}(\vp_{\lambda})|=\sup_{\vp_{\lambda} \in \mathcal{M}(H_1^{\infty}(\mathbb{D}))}|\vp_{\lambda}(f)|=\sup_{\lambda\in \mathbb{D}}|f(\lambda)|= \|f\|,
\]
and hence $\|f\|=\|\hat{f}\|$, that is, $\Gamma$ is an isometry. We identify $H_1^{\infty}(\mathbb{D})$ with $\widehat{H_1^{\infty}(\mathbb{D})}:=\Gamma(H_1^{\infty}(\mathbb{D}))$. Let $\vp_1 \neq \vp_2$ be in $\clm(H^\infty_1(\bd))$. Then there exists $f_0 \in H_1^{\infty}(\mathbb{D})$ such that $\vp_1(f_0) \neq \vp_2(f_0)$, that is
\[
\hat{f_0}(\vp_1)\neq \hat{f_0}(\vp_2).
\]
Therefore $\widehat{H_1^{\infty}(\mathbb{D})}$ separates the points of $\mathcal{M}(H_1^{\infty}(\mathbb{D}))$. Hence we conclude that $H_1^{\infty}(\mathbb{D})$ is an uniform algebra. The remainder of the proof now follows from \cite[Theorem 3]{DRW} and Theorem \ref{T2}.
\end{proof}

For each $n \in \mathbb{N}$, define the algebra $H_{0,1,2,\ldots,n}^{\infty}(\mathbb{D})$ as
\[
H_{0,1,2,\ldots,n}^{\infty}(\mathbb{D}) = \{f \in H^\infty(\bd): f^{(j)}(0)=0, j=0, 1, \ldots, n\}.
\]
Suppose $T$ is an automorphism on $H_{0,1,2,\ldots,n}^{\infty}(\mathbb{D})$ onto itself. A similar argument to the one used to prove the preceding theorem implies that
\[
Tf(z) = f(e^{i \theta} z) \qquad (f\in {H}^{\infty}_{0,1,2, \ldots, n}(\mathbb{D})),
\]
for some $\theta \in \mathbb{R}$. A similar statement as in Corollary \ref{C3} also holds true for surjective linear isometries of $H_{1,2,\ldots,n}^{\infty}(\mathbb{D})$.

\section{Generalized projections on $H^{p}(\mathbb{D})$}\label{sec: GP on D}

In this section, we characterize generalized tri-circular projections on $H^{p}(\mathbb{D})$, $1\leq p \leq \infty$, $p\neq 2$. As part of the necessary background, we require two results from the literature. The first one concerns representations of generalized tri-circular projections \cite[Lemma 1.5]{CJD}:

\begin{lemma}\label{lemma: GTP class}
Let $X$ be a Banach space, $P, Q, R, T \in \clb(X)$, and let $\lambda, \mu \in \mathbb{C} \setminus \{1\}$ be distinct scalars. The following conditions are equivalent:
\begin{enumerate}
\item $T = P + \lambda Q + \mu R$ and $P$, $Q$, and $R$ are projections satisfying $P \oplus Q \oplus R = I$.
\item The following holds: $(T - I)(T - \lambda I)(T - \mu I) = 0$ and
\[
P = \frac{(T - \lambda I)(T - \mu I)}{(\lambda - 1)(\mu - 1)}, Q = \frac{(T - I)(T - \mu I)}{(\lambda - 1)(\lambda - \mu)}, R = \frac{(T - I)(T - \lambda I)}{(\mu - 1)(\mu - \lambda)}.
\]
\end{enumerate}
\end{lemma}

The second tool is classifications of surjective isometries on Hardy spaces \cite[Proposition 2]{NSM}:

\begin{proposition}\label{prop: Nandlal}
Let $1\leq p<\infty$, $p\neq 2$, and let $T \in \clb(H^{p}(\mathbb{D}))$. Then $T$ is a linear surjective isometry if and only if there exists $\tau \in Aut(\bd)$ and a unimodular constant $\alpha$ such that
\[
Tf = \alpha(\tau')^{\frac{1}{p}}f \circ \tau \qquad (f\in H^p({\mathbb{D}})).
\]
\end{proposition}

In what follows, for any $\tau \in Aut(\bd)$ we denote
\[
\tau_0 = (\tau')^{\frac{1}{p}}, \; \tau_1 = (\tau' \circ \tau)^{\frac{1}{p}}, \text{ and } \tau_2 = (\tau' \circ \tau^2)^{\frac{1}{p}}.
\]
Moreover, define $id \in Aut(\bd)$ by
\[
id(z) = z \qquad (z \in \bd).
\]
We are now ready for the classification of generalized tri-circular projections. The Lagrange polynomials are an integral part of the proof presented below, which is also typical for comparable results in other Banach spaces.

\begin{theorem}\label{Th1}
Let $1\leq p \leq \infty$, $p\neq 2$. $P \in \clb(H^{p}(\mathbb{D}))$ is a generalized tri-circular projection if and only if there exists a surjective linear isometry $T \in \clb(H^{p}(\mathbb{D}))$ such that
\begin{enumerate}
\item $T^3 = I$, and
\item $T = P + \lambda Q + \lambda^2 R$ for some nontrivial projection $Q, R \in \clb(H^{p}(\mathbb{D}))$ and a cube root of unity $\lambda$.
\end{enumerate}
Moreover, $P = \frac{1}{3}(I + T + T^2), Q = \frac{1}{3}(I + \lambda^2 T + \lambda T^2)$, and $R =  \frac{1}{3}(I + \lambda T + \lambda^2 T^2)$.
\end{theorem}
\begin{proof}
First, we assume that $1\leq p < \infty$. Suppose $P \in \clb(H^{p}(\mathbb{D}))$ is a generalized tri-circular projection. By the definition, there exist distinct scalars $\lambda_1, \lambda_2\in \mathbb{T}\setminus \{1\}$ and nonzero projections $Q$ and $R$ on $H^{p}(\mathbb{D})$ such that $P \oplus Q \oplus R = I$ and $T:=P + \lambda_1 Q + \lambda_2 R$ is a surjective linear isometry. By Lemma \ref{lemma: GTP class}, we write
\begin{equation}\label{g1}
P = \frac{(T - \lambda_{1}I) (T - \lambda_2I)}{(1 - \lambda_1)(1 - \lambda_2)}, Q = \frac{(T - I)(T - \lambda_2I)}{(\lambda_1 - 1)(\lambda_1 - \lambda_2)}, R = \frac{(T - I)(T - \lambda_1 I)}{(\lambda_2 - 1)(\lambda_2 - \lambda_1)}
\end{equation}
and
\begin{equation}\label{g2}
T^3 - (1 + a)T^2 + (a + b)T - bI = 0,
\end{equation}
where
\[
a = \lambda_1 + \lambda_2 \text{ and } b = \lambda_1\lambda_2.
\]
By Proposition \ref{prop: Nandlal}, there exist $\tau \in Aut(\bd)$ and a unimodular constant $\alpha$ such that
\[
Tf = \alpha \tau_0 f \circ \tau,
\]
for all $f\in H^p({\mathbb{D}})$. Then, for each $f\in H^p({\mathbb{D}})$, we have
\begin{equation}\label{eqn: T2 and T3}
T^2 f = \alpha^2 \tau_0 \tau_1 f\circ\tau^2 \text{ and } T^3f = \alpha^3 \tau_0 \tau_1 \tau_2 f\circ\tau^3,
\end{equation}
and hence by (\ref{g2})
\begin{align}\label{g3}
\alpha^3 \tau_0 \tau_1 \tau_2 f\circ\tau^3 - (1 + a)\alpha^2 \tau_0 \tau_1 f\circ\tau^2 + (a + b)\alpha \tau_0 f\circ\tau - bf = 0.
\end{align}
We claim that there are only three possible options:
\begin{enumerate}
\item $\tau = id$,
\item $\tau \neq id$ and $\tau^2 = id$,
\item $\tau, \tau^2 \neq id$ and $\tau^3 = id$.
\end{enumerate}
Let us assume the contrary. Suppose $id \neq \tau, \tau^2, \tau^3$. Since $\tau, \tau^2,~ \mbox{and}~ \tau^3$ are analytic functions, there exists $z_0 \in \bd$ such that $\{z_0, \tau(z_0), \tau^2(z_0), \tau^3(z_0)\}$ is a set of distinct scalars. Consider a Lagrange polynomial $L$ such that
\[
L(\tau(z_0)) = L(\tau^2(z_0)) = L(\tau^3(z_0)) = 0 \text{ and } L(z_0) = 1.
\]
Applying \eqref{g3} to $f = L$ and at $z = z_0$, we obtain
\[
\begin{split}
0 & = \tau_0(z_0) \{\alpha^3 \tau_1(z_0) \tau_2(z_0) L(\tau(z_0)^3) - (1+a) \tau_1(z_0) L(\tau(z_0)^2) + (a+b) \alpha L(\tau(z_0))
\\
& \quad + (a+b) \alpha L(\tau(z_0))\} - b L(z_0)
\\
& = -b,
\end{split}
\]
and hence $\lambda_1 \lambda_2 = 0$, which gives a contradiction to the fact that both $\lambda_1$ and $\lambda_2$ are nonzero, and proves the claim.

We will now examine each of the three cases separately. First, we consider the nontrivial case (as the other two cases would be shown to be redundant):

\noindent \textbf{Case I:} $\tau, \tau^2 \neq id$ and $\tau^3 = id$. In particular, for $\tau^3 = id$, \eqref{g3} yields
\begin{align}\label{g4}
\alpha^3 \tau_0 \tau_1 \tau_2 f\circ id - (1 + a)\alpha^2 \tau_0 \tau_1 f\circ\tau^2 + (a + b)\alpha \tau_0 f\circ\tau - bf = 0.
\end{align}
Applying this to $f=1$, we obtain
\begin{align}\label{g5}
\alpha^3 \tau_0 \tau_1 \tau_2 = (1 + a)\alpha^2 \tau_0 \tau_1 - (a + b)\alpha \tau_0 + b,
\end{align}
which, applied to \eqref{g4} further yields
\[
(1+a) \alpha^2 \tau_0 \tau_1 (f \circ id - f \circ \tau^2) - (a+b) \alpha \tau_0 (f \circ id - f \circ \tau) + b (f \circ id - f) = 0.
\]
In particular, if $f = id$, then the above identity implies
\begin{equation}\label{eqn: f=id}
(1+a) \alpha^2 \tau_0 \tau_1 (id - \tau^2) = (a+b) \alpha \tau_0 (id - \tau),
\end{equation}
and for $f = id^2$, it yields
\begin{equation}\label{eqn: f=id2}
(1+a) \alpha^2 \tau_0 \tau_1 (id^2 - (\tau^2)^2) - (a+b) \alpha \tau_0 (id^2 - \tau^2) = 0.
\end{equation}
Applying the first identity to the latter identity, we find
\[
(a+b) \alpha \tau_0 (id - \tau) (id + \tau^2) - (a+b) \alpha \tau_0 (id^2 - \tau^2) = 0.
\]
Since $\alpha \tau_0 \neq 0$, it follows that
\[
(a+b) (id - \tau) (\tau^2 - \tau) = 0.
\]
As we know that $id \neq \tau$ and $\tau^2 \neq \tau$, we finally have that $a+b = 0$, that is
\[
\lambda_1 + \lambda_2 + \lambda_1 \lambda_2 = 0.
\]
Finally, plugging the right side of \eqref{eqn: f=id} into \eqref{eqn: f=id2} and noting the fact that $\alpha \tau_0 \tau_1 \neq 0$, we see that
\[
(1+a) (id - \tau^2)(\tau^2 - \tau) = 0,
\]
and hence $1 + a = 0$. Consequently, $1 + \lambda_1 + \lambda_2 = 0$. This together with $\lambda_1 + \lambda_2 + \lambda_1 \lambda_2 = 0$ imply that $\lambda_1 = \lambda$ and $\lambda_{2} = \lambda^2$, where $\lambda$ is cube root of unity. Therefore, by \eqref{g1} and \eqref{g2}, it follows that $T^3 = I$ and
\[
P = \frac{I + T + T^2}{3},  Q = \frac{I + \lambda^2 T + \lambda T^2}{3}, \text{ and } R =  \frac{I + \lambda T + \lambda^2 T^2}{3}.
\]

\noindent \textbf{Case II:} Let $\tau \neq id$ and $\tau^2 = id$. Note in particular that $\tau_2 = \tau_0$ (recall that $\tau_2 = (\tau' \circ \tau^2)^{\frac{1}{p}}$). By \eqref{g3}, we have
\[
\alpha^3 (\tau_0^2 \tau_1) f \circ \tau - (1+a) \alpha^2 \tau_0 \tau_1 f \circ id + (a+b) \alpha \tau_0 f \circ \tau - bf = 0,
\]
for all $f \in H^p(\bd)$. In particular, if $f = 1$, we have
\begin{equation}\label{eqn: eqn C2 1}
\alpha^3 (\tau_0^2 \tau_1) = (1+a) \alpha^2 \tau_0 \tau_1 - (a+b) \alpha \tau_0 + b,
\end{equation}
and, for $f = id$, we obtain
\begin{equation}\label{eqn: eqn C2 2}
\alpha^3 (\tau_0^2 \tau_1) \tau - (1+a) \alpha^2 \tau_0 \tau_1 id + (a+b) \alpha \tau_0 \tau - b id = 0.
\end{equation}
In view of \eqref{eqn: eqn C2 1}, the latter identity implies
\[
(1+a) \alpha^2 \tau_0 \tau_1 (\tau - id) + b (\tau - id) = 0.
\]
 Since $\tau \neq id$ is analytic we have
\begin{equation}\label{eqn: b in 1}
 b = - (1+a) \alpha^2 \tau_0  \tau_1.
\end{equation}
Again, by \eqref{eqn: eqn C2 1}, we have $(1+a) \alpha^2 \tau_0 \tau_1 = \alpha^3 (\tau_0^2 \tau_1)+(a+b)\alpha \tau_0 - b$. Applying this to \eqref{eqn: eqn C2 2}, we have
\[
\alpha^3 (\tau_0^2 \tau_1) \tau - (\alpha^3 (\tau_0^2 \tau_1) + (a+b) \alpha \tau_0 - b) id + (a+b) \alpha \tau_0 \tau - b id = 0,
\]
that is
\[
\alpha^3 (\tau_0^2 \tau_1) (\tau - id) + (a+b) \alpha \tau_0 (\tau - id) = 0.
\]
Again, analyticity of $\tau \neq id$ gives us
\begin{equation}\label{eqn: alpha tau 0 and 1}
 \alpha^2 \tau_0  \tau_1 = -(a+b).
\end{equation}
This along with $ b = - (1+a) \alpha^2 \tau_0 \tau_1$ yield
\[
(1+a)(a+b) - b = 0.
\]
Simplifying this in view of $a = \lambda_1 + \lambda_2$ and $b = \lambda_1 \lambda_2$, we obtain
\[
(1 + \lambda_{1})(1 + \lambda_{2})(\lambda_1 + \lambda_2) = 0.
\]
Let $\lambda_1 = -1$: Then,
\[
\alpha^2 \tau_0 \tau_1= -(a + b) = -(\lambda_1 + \lambda_2 + \lambda_1\lambda_2) = 1,
\]
and hence, by \eqref{eqn: T2 and T3}, we conclude $T^2f = \alpha^2 \tau_0 \tau_1 f\circ\tau^2 = f$, for every $f\in H^p(\mathbb{D})$, that is, $T^2 = I$. Then \eqref{g1} implies that $R=0$. An analogous calculation results in: $\lambda_2 = -1$ implies $Q=0$, and $\lambda_1 = -\lambda_2$ yields $P = 0$. As a result, this case is redundant.

\noindent \textbf{Case III:} $\tau = id$. It follows that $\tau_0 = \tau_1 = \tau_2 \equiv 1$. The identity in \eqref{g3} yields
\[
\alpha^3 f - (1 + a)\alpha^2 f + (a + b)\alpha f - bf = 0,
\]
for all $f \in H^p(\bd)$. In particular, for $f = 1$, we have $\alpha^3 - (1 + a)\alpha^2 + (a + b)\alpha - b = 0$, which implies that $\alpha = 1, \lambda_1, \lambda_2$ (recall that $a = \lambda_1 + \lambda_2$ and $b = \lambda_1 \lambda_2$).

\noindent If $\alpha = 1$, then \eqref{eqn: T2 and T3} implies that $T = I$, and hence $Q= R=0$. Similarly, if $\alpha = \lambda_1$, then $Tf=\lambda_1f$. Hence, we have $P = 0$. Finally, suppose $\alpha = \lambda_2$. Then $Tf = \alpha f = \lambda_2f$ for every $f\in H^p(\mathbb{D})$. Therefore, $T = \lambda_2 I$, and consequently $P = 0$. As a result, this case is also not feasible. This completes the proof of the theorem for $1 \leq p < \infty$, $p \neq 2$.

\noindent Now we turn to $p=\infty$. In this case, the proof is identical to the previous case, but the structure of surjective linear isometries of $H^{\infty}(\mathbb{D})$ must be used. In other words, we need to use the bounded analytic functions version of Proposition \ref{prop: Nandlal}, which can be found in \cite[Theorem 1]{DRW}: An operator $X \in \clb(H^{\infty}(\mathbb{D}))$ is a surjective linear isometry if and only if there exist $\alpha \in \bt$ and $\tau \in Aut(\bt)$ such that
\[
Xf = \alpha (f\circ \tau) \qquad (f\in H^{\infty}(\mathbb{D})).
\]
In view of the above representation, the proof for $p = \infty$ case is identical to the proof for the previous case, and we omit the details. This completes the proof of the theorem.
\end{proof}

Some of the identities established in the preceding theorem will serve as the foundation for additional calculations in the following section.

\section{Generalized projections on $H^{p}(\mathbb{D}^2)$}\label{sec: gen tri 2-var}

The representations of generalized tri-circular projections on $H^{p}(\mathbb{D})$ as obtained in Theorem \ref{Th1} are common among known generalized tri-circular projections acting on Banach spaces. However, in this section, we will see that the structure of generalized tri-circular projections on $H^p(\bd^2)$ is richer.

In what follows, we always assume that $1\leq p \leq \infty$, $p\neq 2$. First, we recall a classification of surjective linear isometries of $H^p(\bd^2)$ \cite[Theorems 1 and 3]{NSM}: $T \in \clb({H}^p(\bt^2))$ is an isometry if and only if there exist $\tau \in Aut(\bd)$, unimodular function $\sigma \in L^\infty(\bt)$, and $\alpha \in \bt$ such that
\begin{equation}\label{eqn: T in 2 vari}
(Tf)(z, w) = \alpha (\tau'(z))^{\frac{1}{p}} f(\tau(z), w \sigma(z)),
\end{equation}
whenever $1\leq p<\infty$, and, if $p = \infty$, then
\begin{equation}\label{eqn: infty on D2}
(Tf)(z, w) = \alpha f(\tau(z), w \sigma(z)),
\end{equation}
for all $f\in {H}^p(\mathbb{T}^2)$, and $z, w \in \bt$.

To ease notation, as in Section \ref{sec: GP on D}, for $\tau \in Aut(\bd)$ and $\sigma \in L^\infty(\bt)$, define
\[
\tau_0(z) = (\tau'(z))^{\frac{1}{p}}, \; \tau_1(z) = (\tau'\circ \tau (z))^{\frac{1}{p}}, \text{ and } \tau_2(z) = (\tau' \circ \tau^2 (z))^{\frac{1}{p}},
\]
and also
\[
\sigma_1(z) = \sigma \circ \tau(z) \text{ and } \sigma_2(z) = \sigma \circ \tau^2(z),
\]
for all $z \in \bt$. However, in what follows, we consider all the above functions (including both $\tau$ and $\sigma$) in two variables, but as functions of $z$ alone. For simplicity of notation, we often write composition of function $f^2$ instead of $f \circ f$ (whenever it make sense). Now we are ready for the characterizations of generalized tri-circular projections on $H^{p}(\mathbb{T}^2)$.

\begin{theorem}\label{Th2}
Let $P \in \clb(H^{p}(\mathbb{T}^2))$. Then $P$ is a generalized tri-circular projection if and only if there exist nontrivial projections $Q, R \in \clb( H^{p}(\mathbb{T}^2))$ and a surjective linear isometry $T \in \clb( H^{p}(\mathbb{T}^2))$ such that one of the following assertions holds:
\begin{enumerate}
\item  $T = P + \lambda Q + \lambda^2 R$, $T^3 = I$, $\lambda$ is a cube root of unity, and $\begin{cases}
P = \frac{1}{3}(I + T + T^2)\\ Q=\frac{1}{3}(I+\lambda^2T+\lambda T^2)\\ R=\frac{1}{3}(I+\lambda T+ \lambda^2 T^2). \end{cases}$
\item $T = P - Q \pm i R$, $T^4 = I$, and $\begin{cases}P = \frac{1}{4}((1 \pm i)T^2 + 2T +(1 \mp i)I) \\ Q = \frac{1}{4}((1 \mp i)T^2 - 2T +(1 \pm i)I) \\ R = \frac{1}{2}(I - T^2).\end{cases}$
\item  $T = P \pm i Q - R$, $T^4 = I$, and $\begin{cases} P = \frac{1}{4}((1 \pm i)T^2 + 2T +(1 \mp i)I) \\ Q = \frac{1}{2}(I - T^2) \\ R = \frac{1}{4}((1 \mp i)T^2 - 2T +(1 \pm i)I).\end{cases}$
\item $T = P \pm i Q \mp i R$, $T^4 = I$, and $\begin{cases} P = \frac{1}{2}(I + T^2)\\ Q = \frac{1}{4}((-1 \pm i)T^2 \mp 2iT +(1 \pm i)I) \\ R = \frac{1}{4}((-1 \mp i)T^2 \pm 2iT +(1 \mp i)I).
\end{cases}$
\end{enumerate}
\end{theorem}

\begin{proof}
Suppose $1\leq p<\infty$. As in the proof of Theorem \ref{Th1}, there exist distinct scalars $\lambda_1, \lambda_2\in \mathbb{T}\setminus \{1\}$ and non-zero projections $Q, R \in \clb(H^{p}(\mathbb{T}^2))$ such that $P \oplus Q \oplus R = I$ and $T:=P + \lambda_1 Q + \lambda_2 R$ is a surjective linear isometry. Moreover
\begin{equation}\label{gp1}
P = \frac{(T - \lambda_{1}I)(T - \lambda_2I)}{(1 - \lambda_1)(1 - \lambda_2)}, Q = \frac{(T - I)(T - \lambda_2I)}{(\lambda_1 - 1)(\lambda_1 - \lambda_2)}, R = \frac{(T - I)(T - \lambda_1 I)}{(\lambda_2 - 1)(\lambda_2 - \lambda_1)},
\end{equation}
and
\begin{equation}\label{gp2}
T^3 - (1 + a)T^2 + (a + b)T - bI = 0,
\end{equation}
where
\[
a = \lambda_1 + \lambda_2 \text{ and } b = \lambda_1\lambda_2.
\]
By \eqref{eqn: T in 2 vari}, there exist $\tau \in Aut(\bd)$, unimodular function $\sigma \in L^\infty(\bt)$, and $\alpha \in \mathbb{T}$ such that
\[
Tf(z, w) = \alpha (\tau'(z))^{\frac{1}{p}} f(\tau(z), w \sigma(z)),
\]
for all $f\in H^p({\mathbb{T}^2})$ and $z, w \in \bt$. With the notation introduced preceding the statement of this theorem, we have
\begin{equation}\label{eqn: T*1}
Tf = \alpha \tau_0 f(\tau, w \sigma) \qquad (f\in H^p({\mathbb{T}^2})),
\end{equation}
and, it follows that
\begin{equation}\label{eqn: T*2}
T^2f = \alpha^2 \tau_0 \tau_1 f(\tau^2, w \sigma \sigma_1)
\end{equation}
and
\[
T^3f = \alpha^3 \tau_0 \tau_1 \tau_2 f(\tau^3, w \sigma \sigma_1 \sigma_2),
\]
for all $f\in H^p({\mathbb{T}^2})$. The identity in \eqref{gp2} then yields
\begin{align}\label{gp3}
\alpha^3 \tau_0 \tau_1 \tau_2 f(\tau^3, w \sigma \sigma_1 \sigma_2) - (1+a) \alpha^2 \tau_0 \tau_1 f(\tau^2, w \sigma \sigma_1) + (a+b) \alpha \tau_0 f(\tau, w \sigma) - bf = 0,
\end{align}
for all $f \in H^{p}(\mathbb{T}^2)$. As in the proof of Theorem \ref{Th1}, we again conclude about the following three possible cases:
\begin{enumerate}
\item $\tau = id$.
\item $\tau \neq id$ and $\tau^2 = id$.
\item $\tau, \tau^2 \neq id$ and $\tau^3 = id$.
\end{enumerate}

\noindent \textbf{Case 1:} Suppose $\tau, \tau^2 \neq id$ and $\tau^3 = id$. We find by using calculations similar to those used in the proof of Case 1 of Theorem \ref{Th1} that $T^3 = I$ and
\[
P = \frac{1}{3}(I + T + T^2), Q = \frac{1}{3}(I + \lambda^2 T + \lambda T^2), R =  \frac{1}{3}(I + \lambda T + \lambda^2 T^2).
\]

\noindent \textbf{Case 2:} Assume that $\tau \neq id$ and $\tau^2 = id$. Since $\tau^2 = id$, it follows that $\tau_2 = \tau_0$ and $\sigma_2 = \sigma$. Therefore, \eqref{gp3} yields
\begin{align}\label{Cgp}
\alpha^3 \tau_0^2 \tau_1 f(\tau, w \sigma^2 \sigma_1) - (1+a) \alpha^2 \tau_0 \tau_1 f(id, w \sigma \sigma_1) + (a+b) \alpha \tau_0 f(\tau, w \sigma) - bf = 0,
\end{align}
for all $f \in H^{p}(\mathbb{T}^2)$. In particular, if $f = z^m$, $m \geq 0$, then
\begin{align}\label{eqn: zm}
\alpha^3 \tau_0^2 \tau_1 \tau^m - (1+a) \alpha^2 \tau_0 \tau_1 z^m + (a+b) \alpha \tau_0 \tau^m - bz^m = 0,
\end{align}
and hence, for $f = 1$, we obtain
\begin{equation}\label{eqn: f=1}
\alpha^3 \tau_0^2 \tau_1 = (1+a) \alpha^2 \tau_0 \tau_1 - (a+b) \alpha \tau_0 + b.
\end{equation}
This is the identity \eqref{eqn: eqn C2 1} obtained in Case 2 during the proof of Theorem \ref{Th1}. Performing the same computation for Case 2 in the proof of Theorem \ref{Th1} results in
\begin{align}\label{gp21}
\lambda_1 = -1,  \text{ or } \lambda_2 = -1, \text{ or } \lambda_1  = -\lambda_2.
\end{align}
Similarly, if $f = w^m$, $m\geq 0$, then \eqref{Cgp} implies
\begin{equation}\label{eqn: f=wm}
\alpha^3 \tau_0^2 \tau_1 (\sigma^2 \sigma_1)^m - (1+a) \alpha^2 \tau_0 \tau_1 (\sigma \sigma_1)^m + (a+b) \alpha \tau_0 \sigma^m - b = 0.
\end{equation}
If $m=1$, then $\alpha^3 \tau_0^2 \tau_1 (\sigma^2 \sigma_1) - (1+a) \alpha^2 \tau_0 \tau_1 (\sigma \sigma_1) + (a+b) \alpha \tau_0 \sigma - b = 0$, and hence, using \eqref{eqn: f=1}, it follows that
\[
(1+a) \alpha^2 \tau_0 \tau_1 (\sigma^2 \sigma_1 - \sigma \sigma_1) - (a+b)  \alpha \tau_0 (\sigma^2 \sigma_1 - \sigma) + b(\sigma^2 \sigma_1 - 1) = 0.
\]
By \eqref{eqn: b in 1}, we know that $b = - (1+a) \alpha^2 \tau_0 \tau_1$. As we are in the same setting as Case 2 of  proof of Theorem \ref{Th1}, we have
\[
-b (\sigma^2 \sigma_1 - \sigma \sigma_1) - (a+b)  \alpha \tau_0 (\sigma^2 \sigma_1 - \sigma) + b(\sigma^2 \sigma_1 - 1) = 0.
\]
After cancelling similar terms, we finally get to the identity
\begin{align}\label{gp18}
b (\sigma \sigma_1 - 1) = (a+b) \alpha \tau_0 (\sigma^2 \sigma_1 - \sigma).
\end{align}
Consider the identity \eqref{eqn: f=wm} again, this time with $m=2$:
\[
\alpha^3 \tau_0^2 \tau_1 (\sigma^2 \sigma_1)^2 - (1+a) \alpha^2 \tau_0 \tau_1 (\sigma \sigma_1)^2 + (a+b) \alpha \tau_0 \sigma^2 - b = 0.
\]
In view of \eqref{eqn: f=1}, we know $\alpha^3 \tau_0^2 \tau_1 = (1+a) \alpha^2 \tau_0 \tau_1 - (a+b) \alpha \tau_0 + b$, and hence, the above equality yields
\[
(1+a) \alpha^2 \tau_0 \tau_1((\sigma^2 \sigma_1)^2 - (\sigma \sigma_1)^2) - (a+b) \alpha \tau_0 ((\sigma^2 \sigma_1)^2 - \sigma^2) + b((\sigma^2 \sigma_1)^2 - 1) = 0.
\]
Substituting $b = - (1+a) \alpha^2 \tau_0 \tau_1$ (see \eqref{eqn: b in 1}) in the above, we derive
\[
b((\sigma^2 \sigma_1)^2 - (\sigma \sigma_1)^2) + (a+b) \alpha \tau_0 ((\sigma^2 \sigma_1)^2 - \sigma^2) - b((\sigma^2 \sigma_1)^2 - 1) = 0.
\]
When we rearrange and simplify by canceling common terms, we get
\[
b((\sigma \sigma_1)^2 -1) = (a+b) \alpha \tau_0 ((\sigma^2 \sigma_1)^2 - \sigma^2).
\]
Applying \eqref{gp18} to this, we find
\[
b((\sigma \sigma_1)^2 -1) = b  (\sigma \sigma_1 - 1) (\sigma^2 \sigma_1 + \sigma).
\]
The equality simplifies further and finally admits the following form:
\[
(\sigma \sigma_1 - 1) (\sigma \sigma_1 + 1)(\sigma - 1) =0.
\]
Therefore, we conclude that
\begin{align}
\sigma \sigma_1 = -1, \text{ or } \sigma \sigma_1 = 1, \text{ or } \sigma = 1.
\end{align}
This combined with \eqref{gp21} results in nine subcases. We summarise them in three subcases. We proceed in the following manner. First, we recall that (see \eqref{eqn: alpha tau 0 and 1})
\begin{equation}\label{eqn: alpha tau 01}
\alpha^2 \tau_0 \tau_1 = -(a+b).
\end{equation}

\noindent\textsf{Subcase (i). $\sigma \sigma_1 = -1$:} Let $\lambda_1 = -1$. Then $a+b = -1$, and hence \eqref{eqn: alpha tau 01} implies $\alpha^2 \tau_0 \tau_1 = 1$. By \eqref{eqn: T*2}, it follows that
\[
(T^2f) (z,w) = f(z, -w) \qquad (z, w \in \bt),
\]
for all $f\in H^p(\mathbb{T}^2)$. Since $\lambda_1 = -1$, by \eqref{gp1} implies $R = \frac{1}{\lambda_2^2-1}(T^2 - I)$. Then
\[
Rf(z, w) = \frac{1}{\lambda_{2}^{2} - 1}(f(z, -w) - f(z, w)),
\]
from which, we further conclude that
\[
R^2f(z, w) = \frac{2}{(\lambda_{2}^{2} - 1)^2}(f(z, w) - f(z, -w)),
\]
for all $f\in H^p(\mathbb{T}^2)$ and $z, w \in \bt$. Since $R^2 = R$, we have
\[
0 = \big(\frac{2}{\lambda_{2}^{2} - 1} + 1\big)(f(z, w) - f(z, -w)) = \frac{\lambda_{2}^{2} + 1}{\lambda_{2}^{2} - 1}(f(z, w) - f(z, -w)),
\]
that is, $ -(\lambda_{2}^{2} + 1) (Rf)(z,w) = 0$ for all $f$. Equivalently, $\lambda_{2} = \pm i$. If $\lambda_2 = i$, then $T=P - Q + iR$, and consequently $T^2 = P + Q - R$. From these two equations we get $P = \frac{1}{2}(T^2 + T +(1 -i)R) = \frac{1}{4}((1 + i)T^2 + 2T +(1 - i)I)$. The similar conclusions hold for $\lambda_2 = -i$. Using \eqref{gp1} we observe that $Q = \frac{1}{4}((1 \mp i)T^2 - 2T +(1 \pm i)I),~R = \frac{1}{2}(I - T^2)$.

\noindent Next, assume that $\lambda_2 = -1$. Here too, $\alpha^2 \tau_0 \tau_1 = 1$, and a calculation similar to the one above implies that
\[
Qf(z, w) = \frac{1}{\lambda_{1}^{2} - 1}\big(f(z, -w) - f(z, w)\big)
\]
and
\[
Q^2f(z, w) = \frac{2}{(\lambda_{1}^{2} - 1)^2}\big(f(z, w) - f(z, -w)\big),
\]
for all $f\in H^p(\mathbb{T}^2)$ and $z, w \in \bt$, and finally, $(\lambda_{1}^{2} + 1)Q = 0$.  Equivalently, $\lambda_1=\pm i$. If $\lambda_1=i$, then $T=P+iQ-R$ and consequently $T^2=P-Q+R$. We obtain $P = \frac{1}{2}(T^2 + T +(1 -i)Q) = \frac{1}{4}((1 + i)T^2 + 2T +(1 - i)I)$. The similar conclusions hold for $\lambda_1 = -i$. Applying \eqref{gp1} we get $Q = \frac{1}{2}(I - T^2),~R = \frac{1}{4}((1 \mp i)T^2 - 2T +(1 \pm i)I)$.

\noindent Finally, assume that $\lambda_1 = -\lambda_2$. Then $a+b = -\lambda_{1}^{2}$ and hence $\alpha^2 \tau_0 \tau_1 = \lambda_{1}^{2}$.  Hence $T^2f(z, w) = \lambda_{1}^{2}f(z, -w)$. By \eqref{gp1}, we have $P = \frac{1}{1 - \lambda_{1}^{2}} (T^2 - \lambda_1^2 I)$, which implies that
\[
Pf(z, w) = \frac{\lambda_{1}^{2}}{1 - \lambda_{1}^{2}}\big(f(z, -w) - f(z, w)\big),
\]
and
\[
P^2f(z, w) = 2\bigg(\frac{\lambda_{1}^{2}}{1 - \lambda_{1}^{2}}\bigg)^2\big(f(z, w) - f(z, -w)\big),
\]
for all $f\in H^p(\mathbb{T}^2)$ and $z, w \in \bt$. Since $P^2 = P$, it follows that
\[
2\bigg(\frac{\lambda_{1}^{2}}{1 - \lambda_{1}^{2}}\bigg)^2\big(f(z, w) - f(z, -w)\big) = \frac{\lambda_{1}^{2}}{1 - \lambda_{1}^{2}}\big(f(z, -w) - f(z, w)\big),
\]
for all $f\in H^p(\mathbb{T}^2)$ and $z, w \in \bt$. We deduce therefore that $(\lambda_{1}^{2} + 1)P = 0$, and hence $\lambda_{1} = \pm i$. Suppose $\lambda_1 = i$. Then $\lambda_2 = -i$ and hence $T = P + iQ - iR$ and $T^2=P-Q-R$. Moreover, $P = \frac{1}{2}(I + T^2)$. Similarly, if $\lambda_1 = -i$, then $T = P - iQ + iR$ and  $T^2=P-Q-R$, and again $P = \frac{1}{2}(I + T^2)$. Applying \eqref{gp1} we obtain $Q = \frac{1}{4}((-1 \pm i)T^2 \mp 2iT +(1 \pm i)I),~R = \frac{1}{4}((-1 \mp i)T^2 \pm 2iT +(1 \mp i)I)$.

\noindent\textsf{Subcase (ii). $\sigma \sigma_1 = 1$:} Recall from \eqref{eqn: alpha tau 01} that $\alpha^2 \tau_0 \tau_1 = -(a+b)$. If $\lambda_1 = -1$ or $\lambda_2 = -1$, then $(a + b) = - 1$, and hence $\alpha^2 \tau_0 \tau_1 =1$. Then \eqref{eqn: T*2} implies $T^2 = I$, and hence, by \eqref{gp1}, $R = 0$ (if $\lambda_1 = -1$) or $Q = 0$ (if $\lambda_2 = -1$). Similarly, if $\lambda_1 = -\lambda_2$, then $\alpha^2 \tau_0 \tau_1 = \lambda_{1}^{2}$. This implies $T^2 = \lambda_{1}^{2} I$, and hence, by \eqref{gp1}, $P = 0$.

\noindent\textsf{Subcase (iii). $\sigma = 1$:} As in the previous subcase, $\lambda_1 = -1$ or $\lambda_2 = -1$ imply $R = 0$  or $Q = 0$, respectively. If $\lambda_1 = -\lambda_2$, then again $\alpha^2 \tau_0 \tau_1 = \lambda_{1}^{2}$, which implies that $T^2 = \lambda_{1}^{2} I$, and consequently $P = 0$.

\noindent\textbf{Case 3:} Suppose $\tau = id$. Then $\sigma_1 = \sigma_2 = \sigma$ and $\tau_1 = \tau_2 = \tau_0$, and hence \eqref{gp3} yields
\begin{align}\label{gp22}
\alpha^3\tau_0^3 f(\tau, w \sigma^3) - (1 + a) \alpha^2 \tau^2_0 f(\tau, w \sigma^2) + (a + b) \alpha \tau_0 f(\tau, w \sigma) - b f(\tau, w) = 0,
\end{align}
for all $f\in H^p(\mathbb{T}^2)$. In particular, if $f = 1$, then
\[
\alpha^3 \tau_0^3  - (1 + a) (\alpha \tau_0)^2 + (a + b) (\alpha \tau_0) - b = 0,
\]
and hence $\alpha = \tau_0^{-1}, \lambda_1 \tau_0^{-1}, \lambda_2 \tau_0^{-1}$. In addition, we also have three alternatives:
\begin{enumerate}
\item $\sigma = 1$.
\item $\sigma \neq 1$ and $\sigma^2 = 1$.
\item $\sigma \neq 1$, $\sigma^2 \neq 1$, and $\sigma^3 = 1$.
\end{enumerate}
Indeed, if $\sigma(z_0), \sigma(z_0)^2, \sigma(z_0)^3 \neq 1$ for some $z_0 \in \bt$, then there exists a Lagrange polynomial $L$ such that $L(w_0\sigma(z_0)) = L(w_0 \sigma(z_0)^2) = L(w_0 \sigma(z_0)^3) = 0$ and $L(w_0) = 1$. If we set $f(z, w) = L(w)$, $z, w \in \bt$, in \eqref{gp22}, then
\begin{align*}
\alpha^3\tau_0^3 L(w_0 \sigma(z_0)^3) - (1 + a) \alpha^2 \tau^2_0 L(w_0 \sigma(z_0)^2) + (a + b) \alpha \tau_0 L(w_0\sigma(z_0)) - bL(w_0) = 0,
\end{align*}
implies that $b = \lambda_1 \lambda_2 = 0$: a contradiction. We now move on to the three following subcases.

\noindent\textsf{Subcase (i). $\sigma = 1$:} If $\alpha = \tau_0^{-1}$ or $\alpha = \lambda_1 \tau_0^{-1}$, then $T = I$ or $T = \lambda_1I$ respectively. Hence $R = 0 = Q$ or $P = 0 = R$. Similarly, if $\alpha = \lambda_2 \tau_0^{-1}$, then $T = \lambda_2 I$, which implies that $P = 0 = Q$.

\noindent\textsf{Subcase (ii).  $\sigma \neq 1$, $\sigma^2 = 1$:} By \eqref{gp22}, we get
\[
\alpha^3\tau_0^3f(\tau, w\sigma) - (1 + a)\alpha^2 \tau_0^2f(\tau, w) + (a + b)\alpha\tau_0f(\tau, w\sigma) - bf(\tau, w) = 0 \notag.
\]
That is
\[
(\alpha^3\tau_0^3 + (a + b)\alpha\tau_0)f(\tau, w\sigma) - ((1 + a)\alpha^2 \tau_0^2 + b)f(\tau, w)= 0 \notag.
\]
Since $\sigma\neq 1$, there exists $z_0$ such that $\sigma(z_0) \neq 1$. We choose Lagrange polynomials, $f(z, w) = L_i(w)\in H^p(\mathbb{T}^2),~~i=1,2 $ such that $ f(z_0, w_0\sigma(z_0))= L_1(w_0\sigma(z_0)) = 1, f(z_0, w_0) = L_1(w_0) = 0$ and $f(z_0, w_0\sigma(z_0)) = L_2(w_0\sigma(z_0)) = 0, f(z_0, w_0) = L_2(w_0) = 1$. Then we have
\[
\alpha^3\tau_0(z_0)^3 + (a + b)\alpha\tau_0(z_0) = 0,
\]
and
\[
(1 + a)\alpha^2 \tau_0(z_0)^2 + b = 0.
\]
From these two equalities we have
\[
(1 + a)(a + b) - b = (1 + \lambda_1 + \lambda_2)( \lambda_1 + \lambda_2 + \lambda_1\lambda_2) - \lambda_1\lambda_2 = 0.
\]
We get, $\lambda_1 = -1$ or $\lambda_2 = -1$ or $\lambda_1 = -\lambda_2$. We consider this with other three possibilities $\alpha = \tau_0^{-1}, \lambda_1 \tau_0^{-1}, \lambda_2 \tau_0^{-1}$ as follows:

\begin{enumerate}
\item $\alpha = \tau_0^{-1}$: If $\lambda_1 = -1$ or $\lambda_2 = -1$, then $R = 0$ or $Q = 0$. If $\lambda_1 = -\lambda_2$, then $P = I$.
\item $\alpha=\lambda_1\tau_0^{-1}$: Suppose $\lambda_1 = -1$ or $\lambda_1 = -\lambda_2$ imply $R = 0$ or $P = 0$. If $\lambda_2 = -1$, then $Q = I$.
\item $\alpha=\lambda_2\tau_0^{-1}$: Let $\lambda_1 = -1$ then $R = I$. If $\lambda_2 = -1$ or $\lambda_1 = -\lambda_2$, then $Q = 0$ or $P = 0$ respectively.
\end{enumerate}

So these subcases are not possible.

\noindent\textsf{Subcase (iii). $\sigma \neq 1$, $\sigma^2 \neq 1$,  $\sigma^3 = 1$:} By \eqref{gp22}, we get
\begin{align}
\alpha^3\tau_0^{3}f(\tau, w) - (1 + a)\alpha^2 \tau_0^{2}f(\tau, w\sigma^2) + (a + b)\alpha\tau_0f(\tau, w\sigma) - bf(\tau, w)= 0 \notag.
\end{align}
Since $\sigma\neq 1, \sigma^2 \neq 1$, then there exists $z_0$ such that $\sigma(z_0) \neq 1$ and $(\sigma(z_0))^2 \neq 1$. We choose Lagrange polynomials, $f(z, w) = L_i(w)\in H^p(\mathbb{T}^2)$, $i=1,2$, such that
\[
f(z_0, w_0\sigma(z_0)^2)= L_1(w_0\sigma(z_0)^2) = 1, f(z_0, w_0) = L_1(w_0) = 0,
\]
\[
f(z_0, w_0\sigma(z_0)) = L_1(w_0\sigma(z_0)) = 0, ~~\mbox{and}~~ f(z_0, w_0\sigma(z_0))= L_2(w_0\sigma(z_0)) = 1,
\]
\[
 f(z_0, w_0) = L_2(w_0) = 0, f(z_0, w_0\sigma(z_0)^2) = L_2(w_0\sigma(z_0)^2) = 0.
 \]
 Then $-(1 + a)\alpha^2 (\tau_0(z_0))^2 = 0$, and $(a + b)\alpha(\tau_0(z_0)) = 0$. Using these two equations we have, $1 + a = 0$ and $a + b = 0$. We observe that $\lambda_1 = \lambda$ and $\lambda_2 = \lambda^2$. Therefore, by \eqref{g1} and \eqref{g2}, it follows that $T^3 = I$ and
\[
P = \frac{I + T + T^2}{3},  Q = \frac{I + \lambda^2 T + \lambda T^2}{3}, \text{ and } R =  \frac{I + \lambda T + \lambda^2 T^2}{3},
\]
which completes the proof of the theorem for $1 \leq p < \infty$. The argument for $p = \infty$ case is also similar. In this case, one needs to use \eqref{eqn: infty on D2}. This completes the proof of the theorem.
\end{proof}

Given the results presented in this paper, representing surjective linear isometries of vector-valued $H^p$-spaces on the bidisc for all $1\leq p \leq \infty$, $p\neq 2$ is an intriguing problem. However, see \cite{Berkson} for the answer in the scalar case.

\vspace{0.2in}

\noindent\textbf{Acknowledgement:}
The research of the first named author is supported by the Theoretical Statistics and Mathematics Unit, Indian Statistical Institute, Bangalore Centre, India and J.C. Bose Fellowship of SERB (India) of Prof. Gadadhar Mishra. The research of the second named author is supported in part by TARE (TAR/2022/000063), SERB, Department of Science \& Technology (DST), Government of India. The research of the third named author is supported by the NBHM postdoctoral fellowship, Department of Atomic Energy (DAE), Government of India (File No: 0204/3/2020/R$\&$D - II/2445).

\vspace{0.0in}

%\noindent\textit{Acknowledgement:}

\bibliographystyle{amsplain}

\end{document}